\documentclass{amsart}
\usepackage{graphics}
\usepackage{graphicx}
\usepackage{amsfonts,amsmath,mathrsfs}
\begin{document}

 \newtheorem{thm}{Theorem}[section]
 \newtheorem{coro}[thm]{Corollary}
 \newtheorem{lemma}[thm]{Lemma}{\rm}
 \newtheorem{proposition}[thm]{Proposition}

 \newtheorem{defn}[thm]{Definition}{\rm}
 \newtheorem{ass}[thm]{Assumption}
 \newtheorem{remark}[thm]{Remark}
 \newtheorem{ex}{Example}
\numberwithin{equation}{section}
\newcommand{\bbR}{\mathbb{R}}
\newcommand{\bbC}{\mathbb{C}}
\newcommand{\bbZ}{\mathbb{Z}}

\def\la{\langle}
\def\ra{\rangle}
\def\fac{{\rm !}}

\def\x{\mathbf{x}}
\def\z{\mathbf{x}}
\def\p{\mathbf{p}}
\def\P{\mathbb{P}}
\def\S{\mathbf{S}}
\def\h{\mathbf{h}}
\def\m{\mathbf{m}}
\def\y{\mathbf{y}}
\def\bz{\mathbf{z}}
\def\F{\mathcal{F}}
\def\R{\mathbb{R}}
\def\T{\mathbf{T}}
\def\N{\mathbb{N}}
\def\D{\mathbf{D}}
\def\V{\mathbf{V}}
\def\U{\mathbf{U}}
\def\K{\mathbf{K}}
\def\Q{\mathbf{Q}}
\def\W{\mathbf{W}}
\def\M{\mathbf{M}}
\def\oM{\overline{\mathbf{M}}}

\def\C{\mathbb{C}}
\def\P{\mathbb{P}}
\def\Z{\mathbb{Z}}
\def\bZ{\mathbf{Z}}
\def\H{\mathcal{H}}
\def\A{\mathbf{A}}
\def\V{\mathbf{V}}
\def\B{\mathbf{B}}
\def\c{\mathbf{C}}
\def\L{\mathcal{L}}
\def\bS{\mathbf{S}}
\def\H{\mathcal{H}}
\def\I{\mathbf{I}}
\def\Y{\mathbf{Y}}
\def\X{\mathbf{X}}
\def\G{\mathbf{G}}
\def\f{\mathbf{f}}
\def\z{\mathbf{z}}
\def\bv{\mathbf{v}}
\def\y{\mathbf{y}}
\def\d{\hat{d}}
\def\x{\mathbf{x}}
\def\bI{\mathbf{I}}

\def\g{\mathbf{g}}
\def\w{\mathbf{w}}
\def\b{\mathbf{b}}
\def\a{\mathbf{a}}
\def\u{\mathbf{u}}
\def\v{\mathbf{v}}
\def\q{\mathbf{q}}
\def\e{\mathbf{e}}
\def\s{\mathcal{S}}
\def\cc{\mathcal{C}}

\def\tg{\tilde{g}}
\def\tx{\tilde{\x}}
\def\tg{\tilde{g}}
\def\tA{\tilde{\A}}

\def\cX{\overline{\mathbf{X}}}
\def\bell{\boldsymbol{\ell}}
\def\bxi{\boldsymbol{\xi}}
\def\balpha{\boldsymbol{\alpha}}
\def\bbeta{\boldsymbol{\beta}}
\def\bgamma{\boldsymbol{\gamma}}
\def\eeta{\boldsymbol{\eta}}
\def\bpsi{\boldsymbol{\psi}}
\def\supmu{{\rm supp}\,\mu}
\def\supp{{\rm supp}\,}
\def\cd{\mathcal{C}_d}
\def\cok{\mathcal{C}_{\K}}
\def\vol{{\rm vol}\,}
\def\om{\mathbf{\Omega}}
\def\blambda{\boldsymbol{\lambda}}
\def\btheta{\boldsymbol{\theta}}
\def\bphi{\boldsymbol{\phi}}
\def\bpsi{\boldsymbol{\psi}}
\def\bnu{\boldsymbol{\nu}}
\def\bmu{\boldsymbol{\mu}}
\def\bom{\boldsymbol{\Omega}}
\def\tM{\hat{\M}}
\def\tv{\hat{\v}}

\title[Pell's equation, SOS and equilibrium measure]{Pell's equation, sum-of-squares 
and equilibrium measures of a compact set}
\thanks{J.B. Lasserre is supported by the AI Interdisciplinary Institute ANITI  funding through the french program
``Investing for the Future PI3A" under the grant agreement number ANR-19-PI3A-0004. This research is also part of the programme DesCartes and is supported by the National Research Foundation, Prime Minister's Office, Singapore under its Campus for Research Excellence and Technological Enterprise (CREATE) programme.}

\author{Jean B. Lasserre}
\address{LAAS-CNRS and Institute of Mathematics\\
University of Toulouse\\
LAAS, 7 avenue du Colonel Roche\\
31077 Toulouse C\'edex 4, France\\
Tel: +33561336415}
\email{lasserre@laas.fr}

\date{}

\begin{abstract}
We first interpret Pell's equation satisfied by Chebyshev polynomials for each degree $t$,
 as a certain Positivstellensatz, which then yields for each integer $t$,
 what we call a generalized Pell's equation,  satisfied by
 reciprocals of Christoffel functions of ``degree" $2t$, associated with
the equilibrium measure $\mu$ of the interval $[-1,1]$ and the measure $(1-x^2)d\mu$. 
 We next extend this point of view to arbitrary compact basic semi-algebraic set 
 $S\subset\R^n$ and obtain a generalized Pell's equation (by analogy with the interval $[-1,1]$).
 Under some conditions, for each $t$ the equation is satisfied by reciprocals of Christoffel functions of ``degree" $2t$ associated with (i) the  equilibrium measure $\mu$ of $S$ and (ii), measures $gd\mu$ for an appropriate set of generators $g$ of $S$. These equations depend on the particular choice of generators that define the set $S$. In addition to the interval $[-1,1]$, 
 we show that for $t=1,2,3$, the equations are indeed also satisfied for
 the equilibrium measures of the $2D$-simplex, the $2D$-Euclidean unit ball and unit box. Interestingly, 
 this view point connects Pell's equation, orthogonal polynomials, Christoffel functions and equilibrium measures on one side, with sum-of-squares, convex optimization and certificates of positivity in real algebraic geometry on another side.
 \end{abstract}

\maketitle

\section{Introduction}
One goal of this paper is to introduce what we call a \emph{generalized Pell's equation} which, under certains conditions, is satisfied by reciprocals of Christoffel functions associated with (i) the equilibrium measure $\lambda_S$ of a compact basic semi-algebraic set $S\subset\R^n$,
and (ii) associated measures $gd\lambda_S$, $g\in G$, 
for an appropriate set $G$ of generators of $S$. Moreover, checking whether
a chosen set $G$ of generators is appropriate, can be done by solving a sequence of convex optimization problems.

Another goal is to reveal via the path to obtain the result, strong links between orthogonal polynomials, Christoffel functions and equilibrium measures on one side, and certificates of positivity in real algebraic geometry, 
optimization and sum-of-squares, as well as a duality result on convex cones by Nesterov, 
on the other side.  

\subsection{Initial and motivating example}
The starting point is Pell's equation satisfied by
Chebyshev polynomials. Pell's equation\footnote{A multivariate polynomial $F\in\Z[\x]$ is called a multi-variable Fermat-Pell polynomial if there exist polynomials $C,H\in\Z[\x]$ such that $C^2-F\,H^2=1$ or $C^2-F\,H^2=-1$ 
for all $\x$. Then the triple $(C,H,F)$ is a multi-variable solution to Pell's equation; see \cite{pell-2}.}
is a topic in algebraic number theory and for more details (not needed here) the interested reader is referred to e.g. \cite{pell-2,pell-1}.
When looking at this equation with special glasses,
we can interpret this equation as a  Putinar's certificate of positivity on the interval $[-1,1]$, for the constant polynomial equal to $1$.
Then with $g(x)=1-x^2$, the reciprocal of the 
two Christoffel functions 
$\Lambda^\mu_t$ and $\Lambda^{g\cdot\mu}_t$ respectively associated with the equilibrium measure
$d\mu=dx/\pi\sqrt{1-x^2}$ and the measure $g\cdot\mu:=gd\mu$, satisfy the same equation, for every $t$. 
Equivalently, for every integer $t$, the two polynomials
$1/(2t+1)\Lambda^\mu_t$ and $(1-x^2)/(2t+1)\Lambda^{g\cdot\mu}_t$ form a 
\emph{partition of unity} of $[-1,1]$. 

More precisely: In $\R[x]$, let $(T_n)_{n\in\N}\subset\R[x]$ (resp. $(U_n)_{n\in\N}\subset\R[x]$)
be the Chebyshev polynomials of the first kind (resp. of the second kind).
Then 
\begin{equation}
\label{pell}
T_n(x)^2+(1-x^2)\,U_{n-1}(x)^2\,=\,1\,,\quad n=1,\ldots\,.
\end{equation}
In other words, for every integer $n\geq1$, the triple $(T_n,(x^2-1),U_{n-1})$ is a solution to 
the polynomial Pell's equation \cite{pell-2}.

Next, let $d\mu(x)=dx/\pi\sqrt{1-x^2}$, $x\mapsto g(x)=1-x^2$, and denote by $g\cdot\mu$ the measure 
$g\,d\mu=\sqrt{1-x^2}dx/\pi$. The family $(\hat{T}_n)_{n\in\N}$ (resp. $(\hat{U}_n)_{n\in\N}$) with
\[\hat{T}_0=T_0\,;\: \hat{T}_n=\sqrt{2}\,T_n\,,\quad n\geq1\,;\quad \hat{U}_n=\sqrt{2}\,U_n\,,\quad\forall n\in\N\,,\]
is orthonormal w.r.t. $\mu$ (resp. $g\cdot\mu$).
Then \eqref{pell} yields
\begin{equation}
\label{cheby-orthonormal}
\hat{T}_n(x)^2+(1-x^2)\,\hat{U}_{n-1}(x)^2\,=\,2\,,\quad n=1,\ldots\,\end{equation}
and consequently, summing up yields
\begin{equation}
\label{pell-sum-0}
\underbrace{\hat{T}_0^2+\sum_{n=1}^t\hat{T}_n(x)^2}_{\Lambda^\mu_t(x)^{-1}}+(1-x^2)\,
\underbrace{\sum_{n=0}^{t-1}\hat{U}_{n}(x)^2}_{\Lambda^{g\cdot\mu}_{t-1}(x)^{-1}}\,=\,2t+1\,,\quad t=0,1,\ldots
\end{equation}
That is:
\begin{equation}
\label{pell-sum-1}
\Lambda^{\mu}_t(x)^{-1}+(1-x^2)\,\Lambda^{g\cdot\mu}_{t-1}(x)^{-1}\,=\,2t+1\,,\quad t=0,1,\ldots,
\end{equation}
where $\Lambda^\mu_t$ (resp. $\Lambda^{g\cdot\mu}_t$)
 is the Christoffel function of ``degree" $2t$ associated with $\mu$ (resp. $g\cdot\mu$). 

So for each integer $t$, the triple $((\Lambda^\mu_t)^{-1},(x^2-1),(\Lambda^{g\cdot\mu}_{t-1})^{-1})$ satisfies 
what we call a \emph{generalized} polynomial Pell's equation; indeed (i)
$(\Lambda^\mu_t)^{-1}\in\Z[x]$ is a \emph{sum} of $t+1$ squares (in short, an SOS) and not a \emph{single} square, and (ii) after scaling, $(\Lambda^\mu)^{-1}_t/(2t+1)\not\in\Z[x]$
(and similarly 
for $(\Lambda^{g\cdot\mu}_{t-1})^{-1})$. Also observe that the scaled polynomials
$\frac{1}{2t+1}(\Lambda^\mu_t)^{-1}$ and $\frac{1}{2t+1}g\cdot(\Lambda^{g\cdot\mu}_{t-1})^{-1}$ form a partition 
of unity for the interval $[-1,1]$.

The measure $\mu$ is called the \emph{equilibrium measure} associated with the interval $[-1,1]$. 
Next, it turns out that \eqref{pell-sum-1} is in fact a particular case of \cite[Theorem 17.7]{nesterov}
which, rephrased later in the polynomial context by the author in \cite[Lemma 4]{cras-1}, 
states that every polynomial $p\in\R[\x]$ (here the constant polynomial $p=2t+1$) 
in the interior of a certain convex cone, has a distinguished representation in terms of certain 
SOS. Namely, such SOS are reciprocals of Christoffel functions associated with some rather 
``intriguing" linear functional
$\phi_p\in\R[\x]^*$ associated with $p$ (see \cite[Equation (10)]{cras-1}). However  in \cite[Lemma 4]{cras-1} we did not provide any clue 
on what is the link between $p$ and  $\phi_p$.
So when $S=[-1,1]$,
\eqref{pell-sum-1} tells us that this intriguing linear functional $\phi_p$ associated with 
\emph{constant} polynomials $p$, is in fact proportional to the
(Chebyshev) equilibrium measure $dx/\pi\sqrt{1-x^2}$ of the interval $[-1,1]$.

So the message of this introductory example is that we can view the polynomial Pell's equation \eqref{pell} as well as 
its generalization \eqref{pell-sum-1}, as 
algebraic Putinar certificates of increasing degree $t=1,2,\ldots$,
that the constant polynomials ($p=1$ for \eqref{pell} and $p=2t+1$
for \eqref{pell-sum-1}) are positive on the interval $[-1,1]$. 

\subsection{Contribution}

The goal of this paper is (i) to define 
a framework that extends the above point  of view to the broader context of compact basic semi-algebraic sets,
(ii) to provide conditions under which a multivariate analogue of \eqref{pell-sum-1} holds, and (iii)
to show that indeed \eqref{pell-sum-1} holds for $t=1,2,3$ for the $2D$-Euclidean ball, the $2D$-unit box, and the $2D$-simplex.
As we next see, Equation \eqref{pell-sum-1} is particularly interesting
as it links statistics, orthogonal polynomials and equilibrium measures on one side, with convex optimization and duality, sum-of-squares and algebraic certificates of positivity, on another side.

More precisely, with $g_j\in\R[\x]$, $j=1,\ldots m$,  let
\begin{equation}
\label{set-S}
 S\,:=\{\,\x\in\R^n: g_j(\x)\geq0\,,\:j=1,\ldots,m\,\}\,,
\end{equation}
be compact with nonempty interior. 
Our contribution is to investigate an appropriate multivariate analogue for $S$ in \eqref{set-S}
and its equilibrium measure,
of the SOS characterization \eqref{pell-sum-1} for the Chebyshev measure $dx/\pi\sqrt{1-x^2}$ on $[-1,1]$. 
Given $g\in\R[\x]$, let
$t_g:=\lceil\mathrm{deg}(g)/2\rceil$, and let $s(t):={n+t\choose n}$. With $g_0=1$, introduce
$G:=\{g_0,g_1,\ldots,g_m\}$
and for every $t\in\N$, let $G_t:=\{g\in G:t_g\leq t\}$ (when $g\in\R[\x]_2$ for all $g\in G$
then $G_t=G$ for all $t\geq1$). For two polynomials $g,h\in\R[\x]$, 
we sometimes use the notation $g\cdot h$ for their usual product, when needed to avoid ambiguity. 
Given a Borel measure $\phi$ on $S$, denote by $g\cdot\phi$,
$g\in G$, the measure $gd\phi$ on $S$. 
Then define the sets
\begin{eqnarray}
\label{quad-module}
Q(G)&:=&
\{\,\sum_{g\in G}\sigma_g\,g\,;\:\sigma_g\in\Sigma[\x]\,\}\\
\label{quad-module-t}
Q_t(G)&:=&
\{\,\sum_{g\in G} \sigma_g\,g\,;\quad\sigma_g\in\Sigma[\x]\,; \quad\mathrm{deg}(\sigma_g\,g)\leq 2t\,\}\,,\quad t\in\N\,,
\end{eqnarray}
respectively called the quadratic module and $2t$-truncated quadratic module
associated with $\{g_1,\ldots,g_m\}\subset\R[\x]$. ($\Sigma[\x]\subset\R[\x]$ is convex cone of sum-of-squares polynomials (SOS in short).)

(i) We first show that if a Borel probability measure $\phi$
on $S$ (with well-defined Christoffel functions $\Lambda^{g\cdot\phi}_t$, $g\in G$, $t\in\N$)
satisfies
\begin{equation}
 \label{contrib-1}
\frac{1}{\sum_{g\in G_t}s(t-t_g)}\,\sum_{g\in G_t} g\cdot (\Lambda^{g\cdot\phi}_{t-t_g})^{-1}\,=\,1\,,\quad \forall\,t\geq t_0\,,
\end{equation}
for some $t_0\in\N$, and $(S,g\cdot\phi)$ satisfies the Bernstein-Markov property (see \eqref{bernstein-markov} below) for every $g\in G$, then necessarily $\phi$ is the equilibrium measure $\lambda_S$ of $S$ (as defined in e.g. \cite{bedford}). Notice that
\eqref{contrib-1} is the perfect multivariate analogue of the univariate 
\eqref{pell-sum-1} for $S=[-1,1]$ and its equilibrium measure $\phi=dx/\pi\sqrt{1-x^2}$; therefore we propose to name \eqref{contrib-1} 
a \emph{generalized Pell's equation} as it is the analogue of \eqref{pell-sum-1} 
for several polynomials $g$, and the solutions 
$(1/\Lambda^{g\cdot\phi}_t)_{g\in G}$ are sums-of-squares (and not a single square as in the multivariate Pell's equation \cite{pell-2}.)
So in this case, for every $t\geq t_0$, as an element of $\mathrm{int}(Q_t(G)^*)$, the vector of degree-$2t$ moments of the equilibrium measure $\lambda_S$, is strongly related to the constant polynomial 
``$1$" in $\mathrm{int}(Q_t(G))$ (which can be viewed as the density of $\lambda_S$ w.r.t. $\lambda_S$). 
Such a situation is  likely to hold only for specify cases  of sets $S$ (with $S=[-1,1]$ and $\lambda_S=dx/\pi\sqrt{(1-x^2)}$ being the prototype example).

However we also show that in the general case, the vector of degree-$2t$ moments of the equilibrium measure $\lambda_S$, 
is still related to the constant polynomial ``$1$" but in a weaker fashion. Namely, let $\mu_t=p^*_t\lambda_S$
be the probability measure whose density $p^*_t$ w.r.t. $\lambda_S$ is the polynomial in the left-hand-side of 
\eqref{contrib-1} (with $\phi=\lambda_S$). Then $\lim_{t\to\infty}\mu_t=\lambda_S$ for the weak convergence of probability measures. 
That is, asymptotically as $t$ grows, and as a density w.r.t. $\lambda_S$, $p^*_t$ behaves like the constant density ``$1$" when integrating continuous functions against $p^*_t\lambda_S$.

(ii) We next provide an \emph{if and only if} condition on $S$ and its representation \eqref{set-S}
so that indeed, for every $t\geq t_0$, there exists a distinguished linear functional
$\phi^*_{2t}\in\R[\x]^*_{2t}$, positive on $Q_t(G)$, which satisfies 
\begin{equation}
 \label{contrib-2}
 1\,=\,\frac{1}{\sum_{g\in G_t}s(t-t_g)}\, \sum_{g\in G_t} g\cdot (\Lambda^{g\cdot\phi^*_{2t}}_t)^{-1}\,,
 \end{equation}
an analogue of \eqref{contrib-1} with Christoffel functions 
$\Lambda^{g\cdot\phi^*_{2t}}_t$ associated with $\phi^*_{2t}$ and $g\in G_t$. 
Interestingly, this condition which states that 
\begin{equation}
\label{contrib-int}
1\,\in\,\mathrm{int}(Q_t(G))\,,\quad \forall t\in\N\,,\end{equation}
is a question of real algebraic geometry related to a (degree-$t$ truncated) 
quadratic module associated with a set $G$ of generators of $S$. Among all possible sets of generators 
for a given compact semi-algebraic set $S$, those $G$ for which \eqref{contrib-int} holds, deserve to be distinguished. 

 (iii) Next, if condition \eqref{contrib-int}
is satisfied then for every fixed $t$, the moment vector
$\bphi^*_{2t}$ associated with the linear functional $\phi^*_{2t}$ in (ii), is the \emph{unique} optimal solution of a convex optimization problem (with a ``$\log\mathrm{det}$" criterion) which can be solved efficiently via off-the-shelf softwares like e.g. 
CVX \cite{cvx} or Julia \cite{julia}. 
In fact, \eqref{contrib-2} is an algebraic ``certificate" that condition 
\eqref{contrib-int} 
holds, and even more, \eqref{contrib-2} and \eqref{contrib-int} are equivalent. Of course, the larger $t$ is, the larger is the size of the resulting convex optimization problem to solve. 

Moreover, every (infinite sequence) accumulation point $\bphi^*=(\phi^*_{\balpha})_{\balpha\in\N^n}$ 
 of the sequence of finite moment-vectors $(\bphi^*_{2t})_{t\in\N}$ 
 associated with the linear functional $\phi^*_{2t}$,
 is represented by a Borel measure $\phi$ on $S$. Then  $\phi$ satisfies \eqref{contrib-1} 
 if and only if  the whole sequence $(\bphi^*_{2t})_{t\in\N}$ converges to
 $\bphi^*$  \emph{and} the convergence is \emph{finite}.
  That is, there exists $t_0\in\N$ such that
 for every $t\geq t_0$, $\phi$ is a representing measure of $\bphi^*_{2t}$. Equivalently,
 for every $t\geq t_0$, $\bphi^*_{2(t+1)}$ is an \emph{extension} of $\bphi^*_{2t}$.
In addition, if the measure $\phi$ is such that $(S,g\cdot\phi)$
 satisfies the Bernstein-Markov property for all $g\in G$, then necessarily $\phi$ is the equilibrium 
 measure $\lambda_S$ of $S$ (by (i)).

Interestingly, this hierarchy of convex optimization problems
provides a practical numerical scheme (at least for moderate values of $t$) 
to check whether the (unique) optimal solution $\bphi^*_{2(t+1)}$ is an extension of $\bphi^*_{2t}$, 
for an arbitrary fixed $t\in\N$, which should eventually happen if \eqref{contrib-1} has ever to hold for the limit measure $\phi$ associated with the sequence $(\bphi^*_{2t})_{t\in\N}$. 

If $\bphi^*_{2(t+1)}$ is an extension of $\bphi^*_{2t}$ for some $t$, then it is a good indication that indeed \eqref{contrib-1} may hold with $\bphi^*_{2t}$ being moments of $\phi$ (up to degree $2t$). 
On the other hand, if $\bphi^*_{2(t+1)}$ is not an extension of $\bphi^*_{2t}$ then it may be because
(i) there is no limit measure $\phi$ that satisfies \eqref{contrib-1}, or (ii) one must wait for a larger $t$ ($t\geq t_0$)
to see a possible ``extension", or (iii) perhaps $G$ is not an appropriate set of generators of $S$. However, in that case it remains to check whether the limit measure $\phi$ is still the equilibrium measure of $S$, and if not, to detect its distinguishing features.

(iv) Finally, in support that \eqref{contrib-1} may be valid 
for sets $S$ other than $[-1,1]$, we also show that 
for $t=1,2,3$, \eqref{contrib-1} holds when
$S$ is the $2D$-Euclidean unit ball and unit box, as well as the $2D$-simplex, in which case
$\lambda_S$ is proportional to
$dxdy/\sqrt{1-x^2-y^2}$, $dxdy/\sqrt{(1-x^2)(1-y^2)}$,
and $dxdy/\sqrt{x\,y\,(1-x-y)}$, respectively.

\section{Main result}

\subsection{Notation and definitions}
Let $\R[\x]$ denote the ring of real polynomials in the variables $\x=(x_1,\ldots,x_n)$ and $\R[\x]_t\subset\R[\x]$ be its subset 
of polynomials of total degree at most $t$. 
Let $\N^n_t:=\{\balpha\in\N^n:\vert\balpha\vert\leq t\}$
(where $\vert\balpha\vert=\sum_i\alpha_i$) with cardinal 
$s(t)={n+t\choose n}$. Let $\v_t(\x)=(\x^{\balpha})_{\balpha\in\N^n_t}$ 
be the vector of monomials up to degree $t$, 
and let $\Sigma[\x]_t\subset\R[\x]_{2t}$ be the convex cone of polynomials of total degree at most $2t$ which are sum-of-squares (in short SOS).

For a real symmetric matrix 
$\A=\A^T$ the notation $\A\succeq0$ (resp. $\A\succ0$) stands for $\A$ is positive semidefinite (p.s.d.) (resp. positive definite (p.d.)).
The support of a Borel measure $\mu$ on $\R^n$ is the smallest closed set $A$ such that
$\mu(\R^n\setminus A)=0$, and such a  set $A$ is unique. Denote by $\mathscr{C}(S)$ 
the space of real continuous functions on $S$.
%
%
\paragraph{Riesz functional, moment and localizing matrix.} 
With a real sequence $\bphi=(\phi_{\balpha})_{\balpha\in\N^n}$  
(in bold) is associated the \emph{Riesz} linear functional $\phi\in \R[\x]^*$  (not in bold) defined by
\[p\:(=\sum_{\balpha}p_{\balpha}\x^{\balpha})\quad \mapsto \phi(p)\,=\,\langle\bphi,p\rangle\,=\,\sum_{\balpha}p_{\balpha}\,\phi_{\balpha}\,,\quad\forall p\in\R[\x]\,,\]
and the moment matrix $\M_t(\bphi)$
with rows and columns indexed by $\N^n_t$ (hence of size $s(t):={n+t\choose t}$), and with entries
\[\M_t(\bphi)(\balpha,\bbeta)\,:=\,\phi(\x^{\balpha+\bbeta})\,=\,\phi_{\balpha+\bbeta}\,,\quad\balpha,\bbeta\in\N^n_t\,.\]
Similarly given $g\in\R[\x]$ ( $\x\mapsto \sum_{\bgamma} g_{\bgamma}\x^{\bgamma}$),
define the new sequence 
\[g\cdot\bphi\,:=\,(\sum_{\bgamma} g_{\bgamma}\,\phi_{\balpha+\bgamma})_{\balpha\in\N^n}\,,\]
and the localizing matrix associated with $\bphi$ and  $g$,
\[\M_t(g\cdot\bphi)(\balpha,\bbeta)\,:=\,\sum_{\bgamma}g_{\bgamma}\,\phi_{\balpha+\bbeta+\bgamma}\,,
\quad\balpha,\bbeta\in\N^n_t\,.\]
Equivalently, $\M_t(g\cdot\bphi)$ is the moment matrix associated with the new sequence 
$g\cdot\bphi$. The Riesz linear functional $g\cdot\phi$ associated with
the sequence $g\cdot\bphi$ satisfies
\[g\cdot\phi(p)\,=\,\phi(g\cdot p)\,,\quad\forall p\in\R[\x]\,.\]
In particular, for any real symmetric $s(t)\times s(t)$ matrix $\Q$
\begin{equation}
\label{aux}
\phi(g(\x)\,\v_{t}(\x)^T\Q\v_t(\x))\,=\,g\cdot\phi(\v_{t}(\x)^T\Q\v_t(\x))
\,=\,\langle \Q,\M_{t}(g\cdot\bphi)\rangle\,.
\end{equation}
A  real sequence $\bphi=(\phi_{\balpha})_{\balpha\in\N^n}$ has a representing mesure 
if its associated linear functional $\phi$ is a Borel measure on $\R^n$. In this case 
$\M_t(\bphi)\succeq0$ for all $t$; the converse is not true in  general.
In addition, if $\phi$ is supported on  the set $\{\,\x\in\R^n: g(\x)\geq0\,\}$ then necessarily
$\M_t(g\cdot\bphi)\succeq0$ for all $t$.

\paragraph{Christoffel function.} Let 
$\phi\in\R[\x]^*$ be a Riesz functional (not necessarily with a representing measure)
such that $\M_t(\bphi)\succ0$. 
As for Borel measures, we may also define
the (degree-$t$) Christoffel function 
\[\x\mapsto \Lambda^{\phi}_t(\x)^{-1}\,:=\,\v_t(\x)^T\M_t(\bphi)^{-1}\v_t(\x)\,,\quad\forall \x\in\R^n\,,\]
associated with $\phi$. Alternatively, if $(P_{\balpha})_{\balpha\in\N^n}\subset\R[\x]$ is a family of polynomials which are orthonormal with respect to $\phi$, then
\begin{equation}
 \label{ortho-poly-0}
\Lambda^{\phi}_t(\x)^{-1}\,=\,\sum_{\balpha\in \N^n_t}P_{\balpha}(\x)^2\,,\quad\forall\x\in\R^n\,.
\end{equation}
Similarly, if $\M_t(g\cdot\bphi)\succ0$,  we may also define the 
(degree-$t$) Christoffel function
\[\x\mapsto \Lambda^{g\cdot\phi}_t(\x)^{-1}\,:=\,\v_t(\x)^T\M_t(g\cdot\bphi)^{-1}\v_t(\x)\,,\quad\forall \x\in\R^n\,,\]
associated with the Riesz functional $g\cdot \phi$. 

All the above definitions also hold for finite sequences 
$\bphi_{2t}=(\phi_{\balpha})_{\balpha\in\N^n_{2t}}$
and associated Riesz linear functional $\phi_{2t}\in\R[\x]^*_{2t}$. Indeed 
when $t$ is fixed, $\Lambda^{\phi}_t$  and $\M_t(\bphi)$ only depend on the degree $2t$-truncation $\bphi_{2t}$ 
of the infinite sequence $\bphi$. Then the notation $\M_t(\bphi)$ or $\M_t(\bphi_{2t})$ 
(and similarly $\Lambda^{\phi_{2t}}_t$ or $\Lambda^{\phi}_t$)
can be used interchangeably.
Finally, a sequence $\bphi_{2(t+1)}$ is an \emph{extension} of $\bphi_{2t}$ if 
$(\phi_{2(t+1)})_{\balpha}=(\phi_{2t})_{\balpha}$ for all $\balpha\in\N^n_{2t}$, i.e., if
$\bphi_{2t}$ is the restriction of $\bphi_{2(t+1)}$ to all moments up to degree $2t$.

\paragraph{Bernstein-Markov property} A Borel measure $\mu$ supported on a compact set $S\subset\R^n$ satisfies the Bernstein-Markov property 
if there exists a sequence of positive numbers $(M_t)_{t\in\N}$ such that for all 
$t$ and $p\in\R[\x]_t$, 
\begin{equation}
\label{bernstein-markov}
\sup_{\x\in S}\vert p(\x)\vert\,\leq\, M_t\cdot \left(\int_S p^2\,d\mu\right)^{1/2}\,,
\quad\mbox{and}\quad\lim_{t\to\infty}\log(M_t)/t\,=\,0\,\end{equation}
(see e.g. \cite[Section 4.3.3]{book}). The Bernstein-Markov property allows 
qualitative description for asymptotics of the Christoffel function as $t$ grows. When it holds it permits to
establish a strong link between the Christoffel function and Siciak's extremal function
\[\x\mapsto V_S(\x)\,:=\,\sup\left\{\frac{\log\vert p(\x)\vert}{\mathrm{deg}(p)}:\:\Vert p\Vert_S\leq 1\,,\: \mathrm{deg}(p)>0\right\}\,,\quad\x\in\R^n\,.\]
A compact set $S$ is said to be \emph{regular} if its associated Siciak's function is continuous everywhere in 
$\R^n$ (the same definition also extends to $\C^n$; see \cite[Definition 4.4.2, p. 53]{book}). If $S$ is regular and $(S,\mu)$
satisfies the Bernstein-Markov property, then 
\[\mbox{uniformly on compact subsets of $\R^n$:}\quad\lim_{t\to\infty}\frac{1}{2t}\,\log{\Lambda^\mu_t(\x)}\,=\,-V_S(\x)\,.\]

\paragraph{Equilibrium measure.}
The notion of equilibrium measure associated to a given set,
originates from logarithmic potential theory (working in
$\mathbb{C}$ in the univariate case) to minimize some energy functional.
For instance, the equilibrium (Chebsyshev) measure $d\phi:=dx/\pi\sqrt{1-x^2}$ 
minimizes the Riesz $s$-energy functional 
\[\int \int \frac{1}{\vert x-y\vert^s}\,d\mu(x)\,d\mu(y)\,\]
with $s=2$, among all measures $\mu$ equivalent to $\phi$.
Some generalizations have been obtained in the multivariate case 
via pluripotential theory in $\mathbb{C}^n$. 
In particular if $S\subset\R^n\subset\mathbb{C}^n$ is compact then 
the equilibrium measure (let us denote it by $\lambda_S$) is equivalent to
Lebesgue measure on compact subsets of $\mathrm{int}(S)$.
It has an even explicit expression if
$S$ is convex and symmetric about the origin; see e.g. 
Bedford and Taylor \cite[Theorem 1.1]{bedford} and \cite[Theorem 1.2]{bedford}. 
Moreover if $\mu$ is a Borel measure on $S$ and $(S,\mu)$ has the Bernstein-Markov property
\eqref{bernstein-markov}
then the sequence of measures $d\nu_t=\frac{d\mu(\x)}{s(t)\Lambda^\mu_t(\x)}$, 
$t\in\N$, converges to $\lambda_S$
for the weak-$\star$ topology and therefore in particular:
\begin{equation}
 \label{weak-star}
 \lim_{t\to\infty}\int_S\x^{\balpha}\,d\nu_t\,=\,\lim_{t\to\infty}\int_S \frac{\x^{\balpha}\,d\mu(\x)}{s(t)\Lambda^\mu_t(\x)}\,=\,
 \int_S\x^{\balpha}\,d\lambda_S\,,\quad\forall \balpha\in\N^n\,
\end{equation}
(see e.g. \cite[Theorem 4.4.4]{book}). In addition, if a compact $S\subset\R^n$ is regular
then $(S,\lambda_S)$ has the Bernstein-Markov property; see \cite[p. 59]{book}. For a brief account
on equilibrium mesures see the discussion in \cite[Section 4-5, pp. 56--60]{book} while for more detailed
expositions see some of the references indicated there.

\subsection{Brief summary of main results}
In Section \ref{preliminary}, Theorem \ref{th0} shows that if a linear functional
$\phi\in\R[\x]^*$ satisfies the multivariate analogue \eqref{contrib-1} of \eqref{pell-sum-1}
for $S$ in \eqref{set-S}, then under a certain technical assumption,
$\phi$ is necessarily the equilibrium measure $\lambda_S$ of $S$. Corollary \ref{ortho-poly}
shows that \eqref{contrib-1} is also a strong property of orthonormal polynomials associated with $\lambda_S$,
the perfect analogue of \eqref{cheby-orthonormal} for Chebyshev polynomials on $S=[-1,1]$.
As this strong property is not expected to hold for general 
sets $S$ in \eqref{set-S}, we next show in Theorem \ref{weaker} that in general,
the polynomial
\[p^*_t\,:=\,\frac{1}{\sum_{g\in G_t}s(t-t_g)}\sum_{g\in G_t}g\cdot (\Lambda_t^{g\cdot\lambda_S})^{-1}\,,\quad t\in\N\,,\]
associated with $\lambda_S$ (now not necessarily constant (equal to $1$) as in Theorem \ref{th0})
has still a strong property related to the constant polynomial ``$1$". Namely  asymptotically, the sequence of probability measures
$(\mu_t:=p^*_t\lambda_S)_{t\in\N}$ (with densities $p^*_t$ w.r.t. $\lambda_S$) converges to $\lambda_S$ for the weak-$\star$ topology of $\mathscr{M}(S)$. That is, informally, the polynomial density $p^*_t$ 
``behaves" asymptotically like the constant (equal to $1$) density 
when integrating continuous functions against $p^*_t\,\lambda_S$. Hence somehow,
the vector of degree-$2t$ moments of $\lambda_S$ in the convex cone $(Q_t(G))^*$ are still
intimately related to the constant polynomial $1$ in $Q_t(G)$ (but not as directly as in Theorem \ref{th0}). 

Next, in Section \ref{log-det} we still consider again the constant polynomial $1$ and 
in Theorem \ref{th1} we show
that under a simple condition, indeed $1\in\mathrm{int}(Q_t(G)$ for all $t$, and therefore 
there exists a sequence of linear functional $(\phi_{2t})_{t\in\N}$ that satisfies \eqref{contrib-2} for all $t$.
For each $t$, the linear functional $\phi_{2t}$ is the unique optimal solution of a simple convex optimization problem with $\log\mathrm{det}$ criterion to maximize.
(In addition, in the case when $S=\{\x: g(\x)\geq0\}$ for some $g\in\R[\x]$, 
Lemma \ref{reverse} relates solutions to Pell's equation with  $\phi_{2t}$ and $S$.)

In Section \ref{sec:asymptotic} one is concerned with the asymptotic behavior of the linear functionals 
$(\phi_{2t})_{t\in\N}$ as $t$ grows, and Theorem \ref{th2} shows that there exists a limit moment sequence 
$\bphi$ which has a representing probability measure $\phi$ on $S$. Moreover $\phi$ satisfies \eqref{contrib-1} and is the equilibrium measure $\lambda_S$, if and only if finite convergence takes place, that is, for every $t\geq t_0$, $\bphi_{2t}$ is the vector of degree-$2t$ moments of $\phi$. So an interesting issue 
(not treated here) is to relate $\phi$ and $\lambda_S$ when the convergence is only asymptotic and not finite.

Finally in Section \ref{eq:numerics} we provide numerical examples 
of sets $S$ where \eqref{contrib-1} holds at least for $t=1,2,3$.

\subsection{Two preliminary results}
\label{preliminary}
For simplicity of exposition, we will consider sets $S$ in \eqref{set-S} for which
the quadratic polynomial $\x\mapsto R-\Vert\x\Vert^2$ belongs to $Q_1(G)$;
in particular, $S$ is contained in the Euclidean ball of radius $\sqrt{R}$ for some $R>0$, and 
the quadratic module $Q(G)$ is Archimedean; see e.g. \cite{CUP}. Let $\lambda_S$ be the equilibrium measure of $S$ (as described in e.g. \cite{bedford}) and recall that $g_0=1$ (so that $g_0\cdot\lambda_S=\lambda_S$). Let $\mathscr{C}(S)$ be the space of continuous functions on $S$.

\begin{ass}
\label{ass-1}
The set $S$ in \eqref{set-S} is compact with nonempty interior. Moreover,
there exists $R>0$ such that the 
quadratic polynomial $\x\mapsto \theta(\x):=R-\Vert\x\Vert^2$ is an element of $Q_1(G)$.
In other words, $h\in Q_1(G)$ is an ``algebraic certificate" that $S$ in \eqref{set-S} is compact.
\end{ass}

\begin{thm}
\label{th0}
With $S$  as in \eqref{set-S}, let Assumption \ref{ass-1} hold. 
Let $\bphi=(\phi_{\balpha})_{\balpha\in\N^n}$ (with $\phi_0=1$) be such that $\M_t(g\cdot\bphi)\succ0$ for all $t\in\N$ and all $g\in G$, so that the Christoffel functions $\Lambda^{g\cdot\phi}_t$ are all well defined
(recall that $\phi\in\R[\x]^*$ is the Riesz  linear functional associated with the moment sequence
$\bphi$). In addition, suppose that there exists $t_0\in\N$ such that
 \begin{equation}
 \label{th0-1}
 1\,=\,
 \frac{1}{\sum_{g\in G_t}s(t-t_g)}\,\sum_{g\in G_t} g\cdot(\Lambda^{g\cdot\phi}_{t-t_g})^{-1}\,,\quad\forall t\geq t_0\,.
 \end{equation}
  \indent Then $\phi$ is a Borel measure on $S$ and the 
 unique representing measure of $\bphi$. Moreover, if 
 $(S,g\cdot\phi)$ satisfies the Bernstein-Markov property for every $g\in G$, then 
 $\phi=\lambda_S$ and  therefore the Christoffel polynomials $(\Lambda^{g\cdot\lambda_S}_{t})^{-1}_{g\in G_t}$ satisfy the generalized Pell's equations:
  \begin{equation}
 \label{th0-2}
 1\,=\,\frac{1}{\sum_{g\in G_t}s(t-t_g)}\,\sum_{g\in G_t} g\cdot(\Lambda^{g\cdot\lambda_S}_{t-t_g})^{-1}\,,\quad\forall t\geq t_0\,.
 \end{equation}
 \end{thm}
 \begin{proof}
 In view of Assumption \ref{ass-1}, the quadratic module $Q(G)$ is Archimedean.
 Next, as $\M_t(g\cdot\bphi)\succ0$ for all $t\in\N$ and all $g\in G$, 
   then by Putinar's Positivstellensatz \cite{putinar}, $\bphi$ has a unique representing measure 
 on  $S$; that is, the Riesz linear functional $\phi$ associated with $\bphi$ is a Borel measure on $S$.  Next, write \eqref{th0-1} as
 \begin{equation}
 \label{aux0-1}
 1\,=\,
 \sum_{g\in G_t}g\cdot\frac{(\Lambda^{g\cdot\phi}_{t-t_g})^{-1}}{s(t-t_g)}\cdot\frac{s(t-t_g)}{\sum_{g\in G}s(t-t_g)}\,,\quad\forall t\geq t_0\,,\end{equation}
 and  let $\balpha\in\N^n$ be fixed arbitrary.
 As $(S,g\cdot\phi)$ satisfies the Bernstein-Markov property for every $g\in G$, then  by \cite[Theorem 4.4.4]{book}, 
   \[\lim_{t\to\infty}\int_S\x^{\balpha}\frac{(\Lambda^{g\cdot\phi}_t)^{-1}}{s(t)}\,g\,d\phi\,=\,\int_S \x^{\balpha}\,d\lambda_S\,,\quad\forall g\in G\,,\]
   where $\lambda_S$ is the equilibrium measure of $S$; see \cite{bedford,book}. Hence multiplying 
   \eqref{aux0-1} by $\x^{\balpha}$ and integrating w.r.t. $\phi$ yields
 \[
 \int_S\x^{\balpha}d\phi\,=\,
 \sum_{g\in G_t} \frac{s(t-t_g)}{\sum_{g\in G_t}s(t-t_g)}\cdot\int_S
 \frac{\x^{\balpha}\cdot(\Lambda^{g\cdot\phi}_{t-t_g})^{-1}}{s(t-t_g)}\,g\,d\phi\,,\quad\forall t\geq t_0\,.\]
 Each term of the product in the above sum
 of the right-hand-side has a limit as $t$ grows. Moreover $G_t=G$ for $t$ sufficiently large.
 Therefore taking limit as $t$ increases yields
  \begin{eqnarray*}
  \int_S\x^{\balpha}d\phi
  &=&\sum_{g\in G} \lim_{t\to\infty}\frac{s(t-t_g)}{\sum_{g\in G}s(t-t_g)}\cdot
  \lim_{t\to\infty}\int_S
 \frac{\x^{\balpha}\cdot(\Lambda^{g\cdot\phi}_{t-t_g})^{-1}}{s(t-t_g)}\,g\,d\phi\\
    &=&\int_S\x^{\balpha}\,d\lambda_S\cdot\sum_{g\in G}
  \lim_{t\to\infty}\sum_{g\in G} \frac{s(t-t_g)}{\sum_{g\in G}s(t-t_g)}\\
  &=&\int_S\x^{\balpha}\,d\lambda_S\cdot\sum_{g\in G}(\#G)^{-1}\,=\,\int_S\x^{\balpha}\,d\lambda_S\,.
    \end{eqnarray*}
  As $\balpha\in\N^n$ was arbitrary and $S$ is compact, then necessarily $\phi=\lambda_S$.
 \end{proof}
Theorem \ref{th0} reveals a strong property of Christoffel functions $\Lambda^{g\cdot\lambda_S}_t$ and is likely to hold 
only in some specific cases. The prototype example is $S=[-1,1]=\{\,x: g(x)\geq0\,\}$ with $x\mapsto g(x)=1-x^2$.  
Then indeed \eqref{pell-sum-1} is exactly \eqref{th0-2}, and by analogy with the Chebyshev univariate case, 
we propose to call Equation \eqref{th0-2} a \emph{generalized Pell's (polynomial) equation} of degree $2t$.  
It is satisfied by the polynomials $(g\cdot(\Lambda^{g\cdot\lambda_S}_{t-t_g})^{-1})_{g\in G_t}$, all of degree less than $2t$.
If true for all $t$, then $(S,\lambda_S)$ satisfies the generalized Pell's equations for all degrees.\\

Of course, to be valid \eqref{th0-2} requires conditions 
on $S$ and its representation \eqref{set-S} by the polynomials $g\in G$. 
For instance, as shown in in Section \ref{eq:numerics} below,  if $S$ is the $2D$-Euclidean unit ball with $g=1-\Vert\x\Vert^2$,
(in which case $G_t=G_1$ for all $t\geq1$), then $\lambda_S=d\x/(\pi\sqrt{1-\Vert\x\Vert^2})$ 
and we can show that \eqref{th0-2} holds for $t=1,2,3$.
Similarly, if $S$ is the $2D$-simplex $\{\x: x_1,x_2\geq0; x_1+x_2\leq1\}$, then
$\lambda_S=d\x/(\pi\sqrt{x_1\cdot x_2\cdot (1-x_1-x_2})$ and
we can show that \eqref{th0-2} holds  for $t=1,2,3$, for the quadratic generators in $G=\{g_0,g_1,g_2,g_3\}$
with $g_1(\x)=x_1\cdot(1-x_1-x_2)$, $g_2(\x)=x_2\cdot(1-x_1-x_2)$, and $g_3(\x)=x_1\cdot x_2$.\\

However in the general case we have the following weaker result, still related to Theorem \ref{th0}.
\begin{thm}
\label{weaker}
Let $\lambda_S$ be the equilibrium measure of $S$ and assume that for every $g\in G$,
$(S,g\cdot\lambda_S)$ satisfies the Bernstein-Markov property. For every $t$, define
the polynomial
\begin{equation}
\label{p-star}
p^*_t\,:=\,\frac{1}{\sum_{g\in G_t}s(t-t_g)}\sum_{g\in G_t}g\cdot (\Lambda_t^{g\cdot\lambda_S})^{-1}\,,\quad t\in\N\,.\end{equation}
Then the sequence of probability measures $(\mu_t:=p^*_t\lambda_S)_{t\geq t_0}$ converges to $\lambda_{S}$ for the weak-$\star$ topology of
$\mathscr{M}(S)$, i.e.,
\begin{equation}
 \lim_{t\to\infty}\int_S f\,p^*_t\,d\lambda_S\,=\,\int_S f\,d\lambda_S\,,\quad\forall f\in\mathscr{C}(S)\,.
\end{equation}
\end{thm}
\begin{proof}
The polynomial $p^*_t$ in \eqref{p-star} is well defined because the matrices $\M_{t-t_g}(g\cdot\lambda_S)$ are non singular.
Each $\mu_t$ is a probability measure on $S$ because
\begin{eqnarray*}
\int p^*_t\,d\lambda_{S}&=&
\frac{1}{\sum_{g\in G_t}s(t-t_g)}\sum_{g\in G_t}\int g\cdot (\Lambda_t^{g\cdot\lambda_{S}})^{-1}\,d\lambda_{S}\\
&=&\frac{1}{\sum_{g\in G_t}s(t-t_g)}\sum_{g\in G_t}\langle \M_{t-t_g}(g\cdot\lambda_{S}),\M_{t-t_g}(g\cdot\lambda_{S})^{-1}\rangle\,\\
&=&\frac{1}{\sum_{g\in G_t}s(t-t_g)}\sum_{g\in G_t}s(t-t_g)\,=\,1\,.
\end{eqnarray*}
As $(S,g\cdot\lambda_S)$ satisfies the Bernstein-Markov property for every $g\in G$, then  by \cite[Theorem 4.4.4]{book}, 
   \[\lim_{t\to\infty}\int_S f\,\frac{(\Lambda^{g\cdot\lambda_S}_t)^{-1}}{s(t)}\,g\,d\lambda_S\,=\,\int_S f\,d\lambda_S\,,\quad\forall f\in\mathscr{C}(S)\,,\:\forall g\in G\,.\]
Hence multiplying    \eqref{p-star} by $f\in\mathscr{C}(S)$ and integrating w.r.t. $\lambda_S$, yields
 \[ \int_S f\,d\mu_t\,=\,
 \int_S f\,p^*_t\,d\lambda_S\,=\,
 \sum_{g\in G_t} \frac{s(t-t_g)}{\sum_{g\in G_t}s(t-t_g)}\cdot\int_S
 \frac{f\cdot(\Lambda^{g\cdot\lambda_S}_{t-t_g})^{-1}}{s(t-t_g)}\,g\,d\lambda_S\,,\quad\forall t\geq t_0\,.\]
 Each term of the product in the above sum
 of the right-hand-side has a limit as $t$ grows. Moreover $G_t=G$ for $t$ sufficiently large.
 Therefore taking limit as $t$ increases, yields
  \begin{eqnarray*}
  \lim_{t\to\infty}\int_S f\,d\mu_t&=&
 \lim_{t\to\infty} \int_S f\,p^*_t\,d\lambda_S\\
  &=&\sum_{g\in G} \lim_{t\to\infty}\frac{s(t-t_g)}{\sum_{g\in G}s(t-t_g)}\cdot
  \lim_{t\to\infty}\int_S
 \frac{f\cdot(\Lambda^{g\cdot\lambda_S}_{t-t_g})^{-1}}{s(t-t_g)}\,g\,d\lambda_S\\
    &=&\int_S f\,d\lambda_S\cdot\sum_{g\in G}
  \lim_{t\to\infty}\sum_{g\in G} \frac{s(t-t_g)}{\sum_{g\in G}s(t-t_g)}\\
  &=&\int_S f\,d\lambda_S\cdot\sum_{g\in G}(\#G)^{-1}\,=\,\int_S f\,d\lambda_S\,,\quad\forall f\in\mathscr{C}(S)\,.
 \end{eqnarray*}
As $S$ is compact it implies that the sequence of probability measures $(\mu_t)_{t\in\N}\subset\mathscr{M}(S)_+$ 
converges to $\lambda_{S}$ for the weak-$\star$ topology $\sigma(\mathscr{M}(S),\mathscr{C}(S))$ of $\mathscr{M}(S)$. 
\end{proof}
In other words (and in an informal language), when integrating continuous functions against $\mu_t$, the density 
$p^*_t$ of $\mu_t$ w.r.t. $\lambda_{S}$ behaves asymptotically like the constant (equal to $1$) density.
That is, Theorem \ref{weaker} is a more general (but weaker) version of Theorem \ref{th0}.

\begin{coro}
\label{ortho-poly}
Let $\phi$ be the Borel measure on $S$  in Theorem \ref{th0}, and for each $g\in G$, let $(P^{g\cdot\phi}_{\balpha})_{\balpha\in\N^n}$ be
a family of polynomials, orthonormal with respect to the measure $g\cdot\phi$. Then for every $t\geq t_0+1$:
\begin{eqnarray}
\label{eq:ortho-poly}
 \sum_{g\in G_t} \sum_{\vert\balpha\vert=t-t_g} g\cdot(P^{g\cdot\phi}_{\balpha})^2
 &=&\sum_{g\in G_t}s(t-t_g)-\sum_{g\in G_t}s(t-t_g-1)\\
 \nonumber
&=&\sum_{g\in G_t}{n-1+t-t_g\choose n-1}\,.
\end{eqnarray}
\end{coro}
\begin{proof}
 Recalling \eqref{ortho-poly-0}, for each $g\in G_t$ with $t\geq t_0+1$:
  \begin{eqnarray*}
  (\Lambda^{g\cdot\phi}_{t-t_g})^{-1}\,=\,\sum_{\balpha\in\N^n_{t-t_g}} (P^{g\cdot\phi}_{\balpha})^2&=&
 \sum_{\vert\balpha\vert<t-t_g} (P^{g\cdot\phi}_{\balpha})^2+
 \sum_{\vert\balpha\vert=t-t_g} (P^{g\cdot\phi}_{\balpha})^2\\
 &=& (\Lambda^{g\cdot\phi}_{t-t_g-1})^{-1}+\sum_{\vert\balpha\vert=t-t_g} (P^{g\cdot\phi}_{\balpha})^2\,,
 \end{eqnarray*}
which combined with \eqref{th0-1} yields \eqref{eq:ortho-poly}.
\end{proof}
\begin{remark}
\label{rem-pell}
Observe that \eqref{eq:ortho-poly} which states a property satisfied by orthonormal polynomials 
associated with $g\cdot\phi$, $g\in G_t$, is a multivariate and multi-generator analogue of \eqref{cheby-orthonormal}, the polynomial Pell's equation satisfied by normalized Chebyshev polynomials. However there are several differences between \eqref{eq:ortho-poly} and
\eqref{cheby-orthonormal}.

In \eqref{cheby-orthonormal}, where $G=\{g\}$ with $g=(1-x^2)$ (and so with $t_g=1$), the \emph{triplet} 
$(\widehat{T}_t, -g,\widehat{U}_{t-t_g})$ is a solution to the polynomial Pell equation
$C^2-F\,H^2=1$ which involves single squares $C^2$ and $H^2$ and a single generator $F$.
On the other hand, \eqref{eq:ortho-poly} addresses the 
\emph{multivariate} case with possibly \emph{several} generators $g\in G_t$ and in compact form reads
$\sum_{g\in G_t}g\,C_g=1$
which now involves SOS polynomials $(C_g)_{g\in G_t}$ and several generators $g\in G_t$. 

For instance, in case of a single generator $G=\{g\}$, in compact form \eqref{eq:ortho-poly} reads $C-F\,H=1$ with now  \emph{SOS polynomials} $C,H$ and generator $F=-g$. Writing 
$C=\sum_{i=1}^r C_i^2$ and $H=\sum_{j=1}^q H_j^2$, 
\[1\,=\,C-F\,H\,=\,\sum_{i=1}^r C_i^2- \sum_{j=1}^q F\,H_i^2\,.\]
This is why we think that it is fair to call \eqref{eq:ortho-poly} (as well as \eqref{contrib-1}) a \emph{generalized} polynomial Pell equation where SOS (rather than just single squares) are allowed. 
In fact, even in the univariate Chebyshev case, when summing over $n$, \eqref{pell-sum-0} is a generalized Pell equation in the form $C-F\,H=1$ with SOS polynomials $C,H$. 
\end{remark}
\subsection{A convex optimization problem and its dual}
\label{log-det}
In Theorem \ref{th0} we have taken for granted existence of  a linear functional $\phi$ 
such that its moment sequence $\bphi$ satisfies \eqref{th0-1}. The next issue is: 

\emph{Given a compact set $S$ as in \eqref{set-S}, can we provide such a moment sequence 
$\bphi$?
At least, can we define a numerical scheme which provides finite sequences $(\bphi_{2t})_{t\in\N}$
which ``converge" to such  a desirable $\bphi$ as $t$ grows? }

As we next see, this issue essentially translates to the following simple issue in real algebraic geometry. 
Do we have $1\in\mathrm{int}(Q_t(G))$ for every $t\in\N$? 
If the answer is yes then indeed such a $\bphi$ exists. But then the associated linear functional $\phi$ will satisfy \eqref{th0-1}
only if the convergence is \emph{finite}. Moreover the conditions can be checked by solving a 
sequence of convex optimization problems described in the next section.

With $t_g:=\lceil\mathrm{deg}(g)/2\rceil$, for every $t\in\N$,  consider the two convex optimization problems:
\begin{equation}
 \label{primal}
 \begin{array}{rl}
\rho_t\,=\, \displaystyle\inf_{\bphi_{2t}}&\{\:-\displaystyle\sum_{g\in G_t} \log\mathrm{det}(\M_{t-t_g}(g\cdot \bphi_{2t}))\,:\: \phi_{2t}(1)\,=\,1\,;\\
 \mbox{s.t.}&\M_{t-t_g}(g\cdot\bphi_{2t})\,\succeq\,0\,,\:\forall g\in G_t\,\}\,,
  \end{array}\,.
\end{equation}
and:
\begin{equation}
 \label{dual}
 \begin{array}{rl}
 \rho^*_t\,=\,\displaystyle\sup_{\Q_g}&\{\:\displaystyle\sum_{g\in G_t} \log\mathrm{det}(\Q_g)\,:\:\Q_g\succeq0\,,\:\forall g\in G_t\,;\\
 \mbox{s.t.}& \displaystyle\sum_{g\in G_t} s(t-t_g)\,=\,
 \sum_{g\in G_t}g(\x)\cdot\v_{t-t_g}(\x)^T\Q_g\v_{t-t_g}(\x)\,,\: \x\in\R^n\,\}\,.
 \end{array}\,.
\end{equation}
Problem \eqref{primal} and \eqref{dual} are convex optimization problems.

\begin{thm}
\label{th1}
With $t\in\N$ fixed, Problems \eqref{primal} and \eqref{dual} have same finite optimal value $\rho_t=\rho_t^*$
if and only if
$1\in \mathrm{int}(Q_t(G))$. Then both have 
a unique optimal solution $\bphi^*_{2t}\in\R^{s(2t)}$ and $(\Q^*_g)_{g\in G_t}$ respectively, which satisfy
$\Q_g^*=\M_{t-t_g}(g\cdot\bphi^*_{2t})^{-1}$ for all $g\in G_t$. Therefore
\begin{eqnarray}
 \nonumber
1&=&\frac{1}{ \displaystyle\sum_{g\in G_t} s(t-t_g)}\,
 \sum_{g\in G_t}g(\x)\,\v_{t-t_g}(\x)^T\M_{t-t_g}(g\cdot\bphi^*_{2t})^{-1}\v_{t-t_g}(\x)\\
 \label{eq:th1-1}
  &=&\frac{1}{ \displaystyle\sum_{g\in G_t} s(t-t_g)}\, \sum_{g\in G_t}g(\x)\,\Lambda^{g\cdot\bphi^*_{2t}}_{t-t_g}(\x)^{-1}\,,\quad\forall  \x\in\R^n\,.
\end{eqnarray}
\end{thm}
\begin{proof}
 For every fixed $t$,  the convex cone $Q_t(G)$ is a  particular case of the 
 convex cone $K(\bar{q})$ investigated in Nesterov \cite[p. 415, Section 2.2]{nesterov}
 when the functional system $\{v(\x)\}$ in \cite{nesterov} is the set of
 monomials $(\x^{\balpha})_{\balpha\in\N^n_{2t}}$  and 
 the functions $(\bar{q}_1,\ldots,\bar{q}_l)$ are our polynomials $g$ in $G_t$.
 Then 
  \[K(\bar{q}_1,\ldots\bar{q}_l)^*\,=\,Q_t(G)^*\,=\,\{\,\bphi_{2t}: \M_{t-t_g}(g\cdot\bphi)\succeq0\,,\:g\in G_t\,\}\,.\]
  By \cite[Theorem 17.7]{nesterov}  
 \[p\in\mathrm{int}(K(\bar{q}_1,\ldots\bar{q}_l))\quad\mbox{if and only if}\quad 
 p\,=\,\sum_{g\in G_t}g\cdot
 \v_{t-t_g}(\x)^T\M_{t-t_g}(g\cdot\bphi_p)^{-1}\v_{t-t_g}(\x)\,,\]
 for some unique $\bphi_p\in K(\bar{q}_1,\ldots\bar{q}_l)^*$. In addition, 
  letting $\Q_g:=\M_{t-t_g}(g\cdot\phi_p)^{-1}$, $g\in G_t$, 
 the sequence $(\Q_g)_{g\in G_t}$ is the unique solution  of \eqref{dual}, 
 with $p$ instead of $\sum_{g\in G_t}s(t-t_g)$ in the left-hand-side 
of the constraint. Therefore, by \cite[Theorem 17.7]{nesterov}  for the constant polynomial $p=1$,
 \[1\,\in\,\mathrm{int}(Q_t(G))\quad\Leftrightarrow\quad 
  1\,=\,\sum_{g\in G_t}g\cdot
 \v_{t-t_g}(\x)^T\M_{t-t_g}(g\cdot\bphi)^{-1}\v_{t-t_g}(\x)\,,\]
 for some distinguished $\bphi\in Q_t(G)^*$.
 Then as $1\in\mathrm{int}(Q_t(G))$ for every $t$, letting $p$ be the constant polynomial 
 $\sum_{g\in G_t}s(t-t_g)$, one obtains 
    \[\sum_{g\in G_t}s(t-t_g)\,=\,\sum_{g\in G_t}g\cdot
 \v_{t-t_g}(\x)^T\M_{t-t_g}(g\cdot\bphi^*_{2t})^{-1}\v_{t-t_g}(\x)\,,\]
 for some unique $\bphi^*_{2t}\in Q_t(G)^*$,
  and $\Q^*_g:=\M_{t-t_g}(g\cdot\bphi^*_{2t})^{-1}$,  $g\in G_t$, is the unique optimal solution of \eqref{dual}. Next, $\bphi^*_{2t}$ is a feasible solution
  of \eqref{primal}, and
 \[\sum_{g\in G_t}\log\mathrm{det}(\Q^*_g)\,=\,-\sum_{g\in G_t}\log\mathrm{det}(\M_{t-t_g}(g\cdot\bphi^*_{2t}))\,\geq\,\rho_t\,.\]
  We next prove weak duality, i.e., $\rho^*_t\leq\rho_t$, so that 
  $\bphi^*_{2t}$ (resp. $(\Q^*_g)_{g\in G_t}$) is the unique optimal
  solution of \eqref{primal} (resp. \eqref{dual}) and $\rho_t=\rho^*_t$.
 So let $\bphi_{2t}$ (resp. $(\Q_g)_{g\in G_t}$) be an arbitrary feasible solution  of \eqref{primal} (resp. \eqref{dual}). Then by Lemma \ref{lem-logdet}, for every $g\in G_t$,
 \[s(t-t_g)+\log\mathrm{det}(\M_{t-t_g}(g\cdot\bphi_{2t}))+\log\mathrm{det}(\Q_g)\,\leq\,\langle\M_{t-t_g}(g\cdot\bphi_{2t}),\Q_g\rangle\,.\]
  In addition, as $\phi_{2t}(1)=1$
  \begin{eqnarray*}
  \sum_{g\in G_t}s(t-t_g)&=&\phi_{2t}(\sum_{g\in G_t}s(t-t_g))
  \,=\,  \sum_{g\in G_t}\phi_{2t}(g(\x)\,\v_{t-t_g}(\x)^T\Q_g\v_{t-t_g}(\x))\\
 &=& \sum_{g\in G_t}g\cdot\phi_{2t}(\v_{t-t_g}(\x)^T\Q_g\v_{t-t_g}(\x))\\
  &=&\sum_{g\in G_t}\langle \Q_g, \M_{t-t_g}(g\cdot\bphi_{2t})\rangle\quad\mbox{[by \eqref{aux}]}\\
  &\geq&\sum_{g\in G_t}[s(t-t_g)+\log\mathrm{det}(\M_{t-t_g}(g\cdot\bphi_{2t}))+\log\mathrm{det}(\Q_g)\,]\,,
  \end{eqnarray*}
 from which we deduce weak duality, that is,
  \[\sum_{g\in G_t}\log\mathrm{det}(\Q_g)\,\leq\,-\sum_{g\in G_t}\log\mathrm{det}(\M_{t-t_g}(g\cdot\bphi_{2t}))\,.\]
 \end{proof}
So as one can see, \eqref{eq:th1-1} is a multivariate analogue of \eqref{pell-sum-1}.  
Crucial in Theorem \ref{th1} is the condition
$1\in\mathrm{int}(Q_t(G))$ for all $t$. Below is a simple sufficient condition.
\begin{lemma}
\label{cond-int}
Let $S$ be as in \eqref{set-S} with $G=\{g_0,g_1,\ldots,g_m\}$, and let Assumption \ref{ass-1} hold.
Then  $1\in\mathrm{int}(Q_t(G))$ for every $t$.
\end{lemma}
For clarity of exposition the proof is postponed to Section  \ref{appendix}.

\begin{remark}
Let $n=1$ and $S=[-1,1]=\{x\in\R: g(x)\geq0\}$ with $x\mapsto g(x)=1-x^2$.
Then $G=\{g\}$, the unique optimal solution 
$\bphi^*_{2t}$ of \eqref{primal} is the vector of moments up to degree $2t$ of the Chebyshev measure $dx/\pi\sqrt{1-x^2}$ on $[-1,1]$, and \eqref{eq:th1-1} is exactly \eqref{pell-sum-1}.
\end{remark}

\begin{lemma}
\label{reverse}
Let $g\in \R[\x]$ of even degree be fixed, $G:=\{g\}$,  and suppose that there are two polynomials of even degree
 $p\in\mathrm{int}(\Sigma_t)$, and $q\in\mathrm{int}(\Sigma_{t-t_g})$  such that $p+g\,q=1$.
Then there exists a linear functional $\phi\in\R[\x]_{2t}^*$ with $\bphi\in \mathrm{int}(Q_t(G)^*)$ such that
\begin{equation}
\label{reverse-1}
1\,=\,\v_t(\x)^T\Lambda^{\phi}_t(\x)^{-1}\,\v_t(\x)+ g(\x)\,\v_{t-t_g}(\x)^T\Lambda^{g\cdot\phi}_{t-t_g}(\x)^{-1}\,\v_t(\x)\,,
\quad\forall \x\in\R^n\,.
\end{equation} 
In particular with $g\in\R[\x]$ fixed: If there exist polynomials $(C_i,H_i)_{i\in I}\subset\Z[\x]$ that solve Pell's polynomial equation $C^2_i+g\, H_i^2=1$, $i\in I$, and if
$\sum_{i\in I} C_i^2\in\mathrm{int}(\Sigma[\x]_t)$, $\sum_{i\in I} H_i^2\in\mathrm{int}(\Sigma[\x]_{t-t_g})$, then \eqref{reverse-1} holds for some $\bphi\in\mathrm{int}(Q_t(G)^*)$.
\end{lemma}
\begin{proof}
 Let $G:=\{g\}$ and let $Q_t(G)$ be as in \eqref{quad-module-t}.
 As $p\in\mathrm{int}(\Sigma_t)$, and $q\in\mathrm{int}(\Sigma_{t-t_g})$, 
 $1=p+g\,q\in \mathrm{int}(Q_t(G))$ and by \cite[Lemma 4]{cras-1},  \eqref{reverse-1} holds. The second statement is a direct consequence by taking $p=\sum_{i\in I} C_i^2$ and $q=\sum_{i\in I} H_i^2$.
\end{proof}
So Lemma \ref{reverse} states that if the triple $(p,\,g,\,q)$ solve the generalized Pell's equation
$p+g\,q=1$, with $p\in\mathrm{int}(\Sigma_t)$ and $q\in\mathrm{int}(\Sigma_{t-t_g})$, then
$p$ (resp. $q$) is the Christoffel polynomial $(\Lambda^{\phi}_t)^{-1}$ (resp. 
$(\Lambda^{g\cdot\phi}_{t-t_g})^{-1}$) associated with some
linear functional $\phi\in\R[\x]^*_{2t}$ such that $\bphi\in\mathrm{int}(Q_t(G)^*)$.

\subsection{An asymptotic result}
\label{sec:asymptotic}
We now consider asymptotics for the sequence $(\bphi^*_{2t})_{t\in\N}$ obtained in Theorem \ref{th1}, as $t$ grows.

\begin{thm}
\label{th2}
Under Assumption \ref{ass-1}, let $\bphi^*_{2t}$ be an optimal solution of  \eqref{primal}, $t\in\N$, guaranteed to exist by Theorem \ref{th1}. Then:

(i) The sequence $(\bphi^*_{2t})_{t\in\N}$ has accumulation points, and for each converging subsequence $(t_k)_{k\in\N}$, $(\bphi^*_{2t_k})_{k\in\N}$ converges pointwise to the vector
$\bphi=(\bphi_{\balpha})_{\balpha\in\N^n}$ of moments of some probability measure $\phi$ on $S$, that is,
\begin{equation}
\label{th2-1}
\lim_{k\to\infty}(\bphi^*_{2{t_k}})_{\balpha}\,=\,\bphi_{\balpha}\,=\,\phi(\x^{\balpha})\,=\,\int_S \x^{\balpha}\,d\phi\,,\quad\forall \balpha\in\N^n\,.\end{equation}

(ii) A limit probability measure $\phi$ as in (i) satisfies \eqref{th0-1} if and only if the whole sequence
$(\bphi^*_{2t})_{t\in\N}$ converges to $\phi$ and finite convergence takes place.
That is, there exists $t_0$ such that for all $t\geq t_0$,
\begin{equation}
\label{th2-11}
 (\bphi^*_{2t})_{\balpha}\,=\,\bphi_{\balpha}\,=\,\phi(\x^{\balpha})\,=\,\int_S \x^{\balpha}\,d\phi\,,\quad\forall \balpha\in\N^n_{2t}\,,\end{equation}
 and so $\phi$  is a representing measure of $\bphi^*_{2t}$ for all $t\geq t_0$.
 In addition, under the condition of Theorem \ref{th0}, $\phi$ is the equilibrium measure $\lambda_S$ of $S$.
\end{thm}
\begin{proof}
 (i) As $R-\Vert\x\Vert^2\in Q_1(G)$, the set of feasible solutions of \eqref{primal} is compact.
 Indeed, let $\bphi_{2t}$ be feasible for \eqref{primal}. Then as $\phi_{2t}(1)=1$, $R\geq \phi_{2t}(x_i^2)$
 for all $i=1,\ldots,n$. Next in multiplying by $x_i$ (with $i$ arbitrary), 
 $R\,x_i^2-x_i^2\cdot\Vert\x\Vert^2\in Q_2(G)$ and so
 \[R^2\,\geq\,R\,\phi_{2t}(x_i^2)\,\geq\,\phi_{2t}(x_i^2\cdot\Vert\x\Vert^2)\,\geq\,\phi_{2t}(x_i^4)\,\Rightarrow\,R^2\geq\phi_{2t}(\x_i ^4)\,.\]
 Iterating yields $R^t\geq\phi_{2t}(\x_i^{2t})$ for every $i=1,\ldots,n$. Then by \cite[Proposition 2.38, p. 41]{CUP} one obtains $\vert (\phi_{2t})_{\balpha}\vert\leq \max[1,R^t]$ for all $\balpha\in\N^n_{2t}$, 
 and all  $\bphi_{2t}\in Q_t(G)^*$. 
 In fact (and assuming $R\geq1$) we even have
 $\vert (\phi_{2t})_{\balpha}\vert\,\leq\,R^{\vert\balpha\vert/2}$, for all $\vert\balpha\vert\leq2t$,  and all  $\bphi_{2t}\in Q_t(G)^*$. 

 By completing with zeros, the finite sequence $\bphi^*_{2t}$ is viewed as an infinite sequence 
 indexed by  $\N^n$. Then by a standard argument involving scaling and the $\sigma(\ell_\infty,\ell_1)$ weak-$\star$ topology,  the sequence $(\bphi^*_{2t})_{t\in\N}$ has accumulation points
and for each subsequence $(t_k)_{k\in\N}$ converging 
to some $\bphi\in\N^n$, one obtains the pointwise convergence
$\lim_{k\to\infty}\,(\phi^*_{2t_k})_{\balpha}=\phi_{\balpha}$, for every $\balpha\in\N^n$.
Next, let $d\in\N$ and $g\in G$ be fixed, arbitrary. Observe that
$\M_d(\bphi^*_{2t_k})\succeq0$ as a principal submatrix of $\M_{t_k}(\bphi^*_{2t_k})\succeq0$, and similarly
$\M_d(g\cdot\bphi^*_{2t_k})\succeq0$ as a principal submatrix of $\M_{t_k}(g\cdot\bphi^*_{2t_k})\succeq0$ 
(when $k$ is sufficiently large so that $G_{t_k}=G$).
Therefore by the above pointwise convergence, $\M_d(g\cdot\bphi^*_{2t_k})\to \M_d(g\cdot\bphi)\succeq0$ as $k$ increases.   As $Q(G)$ is Archimedean, then by Putinar's Positivstellensatz \cite{putinar}, $\phi$ is a Borel probability measure on $S$ (as $\phi^*_{2t_k}(1)=1$ for all $k$).

 (ii) Let $\phi$ be as in (i) and suppose that $\phi$ satisfies \eqref{th0-1}. Then for each $t\geq t_0$,
 the vector $\bphi_{2t}=(\phi_{\balpha})_{\balpha\in\N^n_{2t}}$ is an optimal solution of \eqref{primal},
 and by uniqueness, $\bphi_{2t}=\bphi^*_{2t}$. That is, $\phi$ is a representing measure for $\bphi^*_{2t}$
 for all $t\geq t_0$. But this implies that $\bphi^*_{2(t+1)}$ is an extension of $\bphi^*_{2t}$ for all $t\geq t_0$,
 and therefore the whole sequence converges to $\bphi$, and the convergence is finite.
  
 Conversely, if finite convergence takes place, that is, if $\bphi^*_{2(t+1)}$ is an extension of $\bphi^*_{2t}$ for all $t\geq t_0$,  then $\bphi$ in (i) is the unique accumulation point
 and its associated measure $\phi$ satisfies \eqref{th0-1}.

Finally, if $(S,g\cdot\phi)$ satisfies the Bernstein-Markov property for all $g\in G$, then by Theorem \ref{th0},
 $\phi=\lambda_S$, which concludes the proof.
  \end{proof}

\begin{remark}
Theorem \ref{th2} provides a simple test to detect whether the set $G$ of generators of $S$ is a good one,
and if so, a numerical scheme to compute moments of the equilibrium measure $\lambda_S$ of $S$. 
Indeed if \eqref{th2-11} has to hold for the equilibrium measure $\lambda_S$, then necessarily, the 
unique optimal solution $\bphi^*_{2(t+1)}$ of \eqref{primal} for $t+1$ must be an extension of 
the unique optimal solution $\bphi^*_{2t}$ of \eqref{primal} for $t$, whenever $t$ is sufficiently large.
So for instance, if one observes that $\bphi^*_2$ is an extension of $\bphi^*_1$
after solving \eqref{primal} for $t=1$ and $t=2$, then it already provides a good indication that
finite convergence may indeed take place.
\end{remark}

\section{Examples on some particular sets $S$}
\label{eq:numerics}
We know that Theorem \ref{th2} holds for the equilibrium measure \\
$\lambda_S=1_{[-1,1]}(x)\frac{dx}{\pi\sqrt{1-x^2}}$ of the interval $S=[-1,1]$.
Next, we first show how \eqref{th2-11} holds at least for $t=1,2,3$ 
in the bivariate case with $S$ being the Euclidean unit box and unit ball, or the canonical simplex,
indeed $\bphi^*_{2}$, $\bphi^*_{4}$, and $\bphi^*_{6}$,
are moment vectors up to degree $2$ and $4$ and $6$,
of the equilibrium measures $dxdy/2\pi\sqrt{1-x^2-y^2}$, $dxdy/\pi^2\sqrt{(1-x^2)(1-y^2)}$, and 
$dxdy/\pi\sqrt{x\cdot y\cdot (1-x-y)}$, respectively. \\

\subsection*{On the Euclidean unit box} ~

Let $S:=[-1,1]^2$ and $\lambda_S=dxdy/\pi^2\sqrt{(1-x^2)(1-y^2)}$.
 With the univariate Chebyshev polynomials $T_n$ of first kind and  $U_n$ of second kind, and letting
 \[g_1(x,y)\,:=\,(1-x^2)\,;\quad g_2(x,y)\,:=\,(1-y^2)\,;\quad g_3(x,y)\,:=\,(1-x^2)(1-y^2)\,,\]
 with $G=\{g_1,g_2,g_3\}$:
 
 - $(P_{ij}(x,y):=\hat{T}_i(x)\hat{T}_j(y))_{i,j\in\N}$ form an orthonormal family with respect to $d\lambda_S$, 
 
 - $(P^{g_1}_{ij}(x,y):=\hat{U}_i(x)\hat{T}_j(y))_{i,j\in\N}$ form an orthonormal family with respect to $(1-x^2)d\lambda_S$, 
 
 -  $(P^{g_2}_{ij}(x,y):=\hat{T}_i(x)\hat{U}_j(y))_{i,j\in\N}$ form an orthonormal family with respect to $(1-y^2)d\lambda_S$,
 
 - $(P^{g_3}_{ij}(x,y):=\hat{U}_i(x)\hat{U}_j(y))_{i,j\in\N}$ form an orthonormal family with respect to $(1-x^2)(1-y^2)d\lambda_S$. 
 
 Then 
 \[P_{10}(x,y)^2+P_{01}^2(x,y)+(1-x^2)\,P^{g_1}_{00}(x,y)^2+(1-y^2)\,P^{g_2}_{00}(x,y)^2\,=\,2\,,\]
 from which we obtain:
 \[(\Lambda^{\lambda_S}_1)^{-1}+g_1\cdot(\Lambda^{g_1\cdot\lambda_S}_0)^{-1}+g_2\cdot(\Lambda^{g_2\cdot\lambda_S}_0)^{-1}\,=\,5\,=\,s(1)+2s(0)\,.\]
 Next,
 \begin{eqnarray*}
 P_{20}(x,y)^2+P_{11}(x,y)^2+P_{02}(x,y)^2&=&8x^4+8y^4-8x^2-8y^2+4x^2y^2+4\\
 (1-x^2)\,((P^{g_1}_{10})^2+(P^{g_1}_{10})^2)&=&8x^2-8x^4+4y^2-4x^2y^2\\
 (1-y^2)\,((P^{g_2}_{10})^2+(P^{g_1}_{10})^2)&=&8y^2-8y^4+4x^2-4x^2y^2\\
  (1-x^2)(1-y^2)\,(P^{g_3}_{00})^2&=&4-4x^2-4y^2+4x^2y^2\,.
  \end{eqnarray*}
 Again after scaling (to get borthonormal polynomials) and summing up, one obtains
 \[(\Lambda^{\lambda_S}_2)^{-1}+g_1\cdot(\Lambda^{g_1\cdot\lambda_S}_1)^{-1}+g_2\cdot(\Lambda^{g_2\cdot\lambda_S}_1)^{-1}
 +g_3\cdot(\Lambda^{g_3\cdot\lambda_S}_0)^{-1}\,=\,13\,=\,s(2)+2s(1)+s(0)\,.\]
\subsection*{On the $2D$-Euclidean ball}~
With $\lambda_S=\frac{dx\,dy}{2\pi\sqrt{1-x^2-y^2}}$, one obtains:
\[\M_2(\lambda_S)\,=\,
\left[\begin{array}{cccccc} 1 & 0 &0 &1/3 &0 &1/3\\
0 & 1/3 &0 &0 &0 &0\\
0 & 0 & 1/3 &0 &0 &0 \\
1/3& 0 &0 &1/5 &0 &1/15 \\
0& 0 &0 &0 &1/15 &0 \\
1/3& 0 &0 &1/15 &0 &1/5\end{array}\right]\,.\]
Similarly, with $g\cdot\lambda_S=(1-x^2-y^2)\,dx\,dy=\sqrt{1-x^2-y^2}\,dx\,dy/2\pi$,
\[\M_1(g\cdot\lambda_S)\,=\,
\left[\begin{array}{ccc} 
1/3 & 0 &0 \\
0 & 1/15 &0 \\
0 & 0 & 1/15
\end{array}\right]\,,\]
which yields
\[\Lambda^{\lambda_S}_1(x,y)^{-1}+(1-x^2-y^2)\cdot\Lambda^{g\cdot\lambda_S}_0(x,y)^{-1}\,=\,3+1\,=\,s(1)
+s(0)\,,\]
as well as
\[\Lambda^{\lambda_S}_2(x,y)^{-1}+(1-x^2-y^2)\cdot\Lambda^{g\cdot\lambda_S}_1(x,y)^{-1}\,=\,9\,=\,s(2)
+s(1)\,,\]
and similarly
\[\Lambda^{\lambda_S}_3(x,y)^{-1}+(1-x^2-y^2)\cdot\Lambda^{g\cdot\lambda_S}_2(x,y)^{-1}\,=\,16\,=\,s(3)+s(2)
+s(1)\,,\]
as indicated in Theorem \ref{th0}.\\

\subsection*{On the simplex}~
Consider the canonical simplex $S=\{(x,y)\in\R^2: x+y\,\leq\,1\,;\: x\,,y\,\geq0\,\}$
with equilibrium measure
\[\lambda_S\,=\,\frac{1_{S}(x,y)dx\,dy}{2\pi\sqrt{x}\,\sqrt{y}\,\sqrt{1-x-y}}\,.\]
There are several ways to represent $S$ and in particular, consider $G=\{g_1,g_2,g_3\}$ with
\[(x,y)\mapsto g_1(x,y)\,:=\,x\cdot (1-x-y)\,;\quad
g_2(x,y)\,:=\,y\cdot (1-x-y)\,,\]
and $(x,y)\mapsto g_3(x,y):=x\cdot y$. Do we have
\[\Lambda^{\lambda_S}_t(x)^{-1}+g_1\cdot\Lambda^{g_1\cdot\lambda_S}_{t-1}(x)^{-1}+g_2\cdot\Lambda^{g_2\cdot\lambda_S}_{t-1}(x)^{-1}+g_3\cdot\Lambda^{g_3\cdot\lambda_S}_{t-1}(x,y)^{-1}\,=\,1\,?\,\]
With $t=1$, the moment matrix of $\lambda_S$ reads:
\[\M_1(\lambda_S)\,=\,\frac{1}{15}\,\left[\begin{array}{ccc}
15 &5 & 5\\
5 &3&1\\
5&1&3
\end{array}\right]\,.\]
Next, we obtain 
\[\int_S x\cdot (1-x-y)\,d\lambda_S\,=\,\int_S y\cdot (1-x-y)\,d\lambda_S\,=\,\frac{1}{15}\,;\quad
\int_S xy\cdot (1-x-y)\,d\lambda_S\,=\,\frac{1}{15}\,.\]
Hence 
\begin{eqnarray*}
\Lambda^{\lambda_S}_1(x,y)^{-1}
&=&15\,(\frac{2}{5}-x-y+x^2+y^2+xy)\\
x\cdot (1-x-y)\,\Lambda^{g_1\cdot\lambda_S}_0(x,y)^{-1}
&=&15\,(x-x^2-xy)\\
y\cdot (1-x-y)\,\Lambda^{g_2\cdot\lambda_S}_0(x,y)^{-1}
&=&15\,(y-y^2-xy)\\
x\cdot y\,\Lambda^{g_3\cdot\lambda_S}_0(x,y)^{-1}&=&15\,xy\,,
\end{eqnarray*}
and therefore
\[\Lambda^{\lambda_S}_1(x,y)^{-1}+x\cdot (1-x-y)\,\Lambda^{g_1\cdot\lambda_S}_0(x,y)^{-1}\]
\[+\,y\cdot (1-x-y)\,\Lambda^{g_2\cdot\lambda_S}_0(x,y)^{-1}+x\cdot y\,\Lambda^{g_3\cdot\lambda_S}_0(x,y)^{-1}=6\,=\,s(1)+3\,s(0)\,,\]
as indicated by Theorem \ref{th0}. 
Similarly,
 \[\M_2(\lambda_S)\,=\,\left[\begin{array}{cccccc}   
 1.0000   & 0.3333   & 0.3333 &   0.2000  &  0.0667 &   0.2000\\
    0.3333  &  0.2000  &  0.0667  &  0.1429  &  0.0286 &   0.0286\\
    0.3333   & 0.0667  &  0.2000  &  0.0286 &   0.0286  &  0.1429\\
    0.2000 &   0.1429  &  0.0286 &   0.1111 &   0.0159  &  0.0095\\
    0.0667  &  0.0286   & 0.0286 &   0.0159 &   0.0095  &  0.0159\\
    0.2000   & 0.0286 &   0.1429   & 0.0095  &  0.0159 &   0.1111\end{array}\right]\,,\]
    whereas
    \[\M_1(g_1\cdot\lambda_S)\,=\,
    \left[\begin{array}{ccc}
    0.0667 & 0.0286  &  0.0095\\
    0.0286  &  0.0159 &   0.0032\\
    0.0095  &  0.0032  &  0.0032\end{array}\right]\,;\quad
    \M_1(g_2\cdot\lambda_S)\,=\,
    \left[\begin{array}{ccc}
    0.0667 & 0.0095 &0.0286  \\
    0.0095  &  0.0032 &   0.0032\\
    0.0286  &  0.0032  &  0.0159\end{array}\right]\,,\]
    and
  \[\M_1(g_3\cdot\lambda_S)\,=\,
    \left[\begin{array}{ccc}  0.0667 &   0.0286 &   0.0286\\
    0.0286  &  0.0159  &  0.0095\\
    0.0286  &  0.0095   & 0.0159\end{array}\right]\,.\]
    Then this implies
    \[(\Lambda^{\lambda_S}_2)^{-1}+\sum_{i=1}^3
    g_i\cdot (\Lambda^{g_i\cdot\lambda_S}_1)^{-1}\,=\,15\,=\,s(2)+3\,s(1)\,.\]
  In continuing with $t=3$ we also obtain
    \[(\Lambda^{\lambda_S}_3)^{-1}+\sum_{i=1}^3
    g_i\cdot (\Lambda^{g_i\cdot\lambda_S}_2)^{-1}\,=\,28\,=\,s(3)+3\,s(2)\,.\]

\subsection*{Intersection of two ellipsoids}~
Here we consider $S\subset \R^2$, $G=\{g_0,g_1,g_2\}$, with
\[g_1(x,y)\,:=\,1-2\,x^2-3\,y^2\,;\quad g_2(x,y)\,:=\,1-3\,x^2-2\,y^2\,.\]
so that $S$ is the intersection of two ellipsoids. With $t=1,2,3$, the respective optimal solutions 
$\bphi^*_2,\bphi^*_4$ and $\bphi^*_6$ of \eqref{primal} 
are such that $\bphi^*_4$ seems to be an extension of $\bphi^*_2$, and $\bphi^*_6$ seems to be an extension of 
$\bphi^*_4$, up to some numerical imprecision due to the solver; indeed the norm
of the difference between $\bphi^*_4$ and the restriction of $\bphi^*_6$ is about $0.006$, and for instance
\[\bphi^*_2\,=\,(1,0,0,0.00999961,0,0.00999962)\,\quad \bphi^*_4\,=\,(1,0,0,0.0117564,0,0.01175, \ldots)\,.\]
\[\bphi^*_4\,=\,(1,0,0,0.0117564,0,0.01175, \ldots)\,;\quad \bphi^*_6\,=\,(1,0,0,0.011506,0,0.0111425, \ldots)\,.\]
However, as the solver is not very accurate, it is difficult to conclude whether or not the differences between $\bphi_2,\bphi_4$ and $\bphi_6$ are due to numerical inaccuracies.

\subsection*{The TV-screen}~
Let $S:=\{\,(x,y)\in\R^2: x^4+y^4\leq1\,\}$. 
By solving numerically \eqref{primal}-\eqref{dual} with $t=2$ and $t=3$, we find
\[\Lambda^{\bphi^*_4}_2(x,y))^{-1}+(1-x^4-y^4)\,\Lambda^{g\cdot\bphi^*_4}_0(x,y)^{-1}\,=\,6+1\,=\,s(2)+s(0)\,\]
\[\Lambda^{\bphi^*_6}_3(x,y)^{-1}+(1-x^4-y^4)\,\Lambda^{g\cdot\bphi^*_6}_1(x,y)^{-1}\,=\,10+3\,=\,s(3)+s(1)\,\]
respectively, as predicted by Theorem \ref{th1}. However, 
we observed that $\bphi^*_6$ is not an extension of $\bphi^*_4$.

\subsection*{The Gaussian case}~
Finally, in the same spirit but not in the preceding context of a compact set $S\subset\R^n$, we 
consider the case of $\R^n$ where by \cite[Lemma 3]{cras-1} any polynomial $p\in\mathrm{int}(\Sigma[\x]_t)$
is the Christoffel function of some linear functional $\phi_p\in\Sigma[\x]^*_t$. 
Let $\Sigma\succ0$ be a real symmetric $n\times n$ matrix, and associated with $\Sigma$,  let $p$ be
the quadratic polynomial
\[\x\mapsto p(\x)\,:=\,1+\langle \x,\Sigma^{-1}\,\x\rangle\,,\quad\x\in\R^n\,,\]
which is in the interior of the convex cone $\Sigma[\x]_1$. Therefore, by \cite{nesterov} and \cite[Lemma 3]{cras-1}
\[p(\x)\,=\,
\v_1(\x)^T\M_1(\bphi)^{-1}\v_1(\x)\,,\quad \forall \x\in\R^n\,,\]
for some unique $\bphi\in\Sigma[\x]_1^*$. It is straightforward to check that
\[p(\x)\,=\,\v_1(\x)^T\left[\begin{array}{cc}1&0\\0&\Sigma\end{array}\right]^{-1}\v_1(\x)\,,\quad\forall \x\in\R^n\,,\]
and $\left[\begin{array}{cc}1&0\\0&\Sigma\end{array}\right]^{-1}$ is the  moment matrix 
of the Gaussian measure 
\[d\phi_p=(\mathrm{det}(2\pi\,\Sigma))^{-1/2}\exp(-\x^T\Sigma^{-1}\x/2)d\x\,.\]
That is, $\bphi$ is represented by the Gaussian measure $\phi_p$ and 
$\Lambda^{\phi}_1$ is the Christoffel  function of degree $2$ of the Gaussian measure $\phi_p$.
\[\Lambda^{\phi_p}_1(\x)^{-1}\,=\,(1,\x)^T\left[\begin{array}{cc}1&0\\0&\Sigma\end{array}\right]^{-1}(1,\x)\,=\,1+\x^T\Sigma^{-1}\x\,=\,p(\x).\]

\subsection{Discussion}

There are several issues that are worth investigating. The first one is to completely
validate our result for $t>3$, for the cases where $S$ is the unit box, the Euclidean unit ball, and the simplex.
One possibility is to use Corollary \ref{ortho-poly} for each degree $t$, which only requires
to show \eqref{eq:ortho-poly} (a property of orthonormal polynomials associated
with the measures $(g\cdot\lambda_S)_{g\in G}$) as we did on some of the above examples.

Another issue is to investigate what is a distinguishing feature of the limit measure $\phi$ in Theorem \ref{th2} when 
$\phi$ does not satisfy the generalized Pell's equation \eqref{contrib-1}. Could $\phi$ still be the equilibrium measure of $S$?

Finally, one would like to extend the present framework and characterize the linear functional $\phi_p\in
Q_t(G)^*$ in \cite[Lemma 4]{cras-1}, associated with a non constant polynomial $p\in\mathrm{int}(Q_t(G))$.
Then with arguments that mimic those used for the constant polynomial $p=1$, if 
$p\in\mathrm{int}(Q_t(G))$ for every $t\in\N$, then a natural candidate 
seems to be $\phi_p:=\lambda_S/p$, that is, $\phi_p$ is the measure with density $1/p$ with respect to the equilibrium measure $\lambda_S$ of $S$.

\section{Appendix}
\label{appendix}
\begin{lemma}
\label{lem-logdet}
Let $\mathcal{S}^n$ be the space of real symmetric $n\times n$ matrices and let $\mathcal{S}^n_{++}\subset\mathcal{S}^n$ 
be the convex cone of real $n\times n$ positive definite matrices $\Q$ (denoted $\Q\succ0$). Then
\begin{equation}
\label{fenchel}
n+\log\mathrm{det}(\M)+\log\mathrm{det}(\Q)\,\leq\,\langle\M,\Q\rangle\,,\quad\forall 
\M\,,\Q\,,\in\,\mathcal{S}^n\,.
\end{equation}
with equality if and only if $\Q=\M^{-1}$.
\end{lemma}
\begin{proof}
  Consider the concave function
\[f:\: \mathcal{S}^n\to \R\cup\{-\infty\}\,\quad \Q\mapsto f(\Q)\,=\,\left\{\begin{array}{l}
\log\mathrm{det}(\Q)\:\mbox{ if $\Q\in\mathcal{S}^n_{++}$,}\\
-\infty\:\mbox{ otherwise,}\end{array}\right.\]
and let $f^*$ be its (concave analogue) of Legendre-Fenchel conjugate, i.e.,
\[\M\,\mapsto f^*(\M)\,:=\,\inf_{\Q\in\mathcal{S}^n}\langle \M,\Q\rangle-f(\Q)\,.\]
It turns out that
\[f^*(\M)\,=\,\left\{\begin{array}{l}n+\log\mathrm{det}(\M)\,(=\,n+f(\M))\mbox{ if $\M\,\in\,\mathcal{S}^n_{++}$ ,}\\
-\infty\mbox{ otherwise.}\end{array}\right.\]
Hence the concave analogue of Legendre-Fenchel inequality states that
\[f^*(\M)+f(\Q)\,\leq\,\langle \M,\Q\rangle\,,\quad\forall \M\,,\Q\,\in\,\mathcal{S}^n\,,\]
and yields  \eqref{fenchel}.
\end{proof}

\paragraph{Proof of Lemma \ref{cond-int}}.
\begin{proof}
 Recall that $\theta(\x)=1-\Vert\x\Vert^2$. By \cite[Lemma 3.4]{ctp}
 \[(R+1)^t=\underbrace{(1+\Vert\x\Vert^2)^t}_{\Delta}+\theta(\x)\,\underbrace{\sum_{j=0}^{t-1}(R+1)^j(1+\Vert\x\Vert^2)^{t-j-1}}_{\Gamma}\,.\]
 Note that 
 \begin{eqnarray*}
 \Delta(\x)&=&\sum_{\balpha\in\N^n_t}\Theta_{\balpha}\x^{2\balpha}\,=\,\v_t(\x)^T\G_0\v_t(\x)\,,\quad\forall \x\,,\\
  \Gamma(\x)&=&\sum_{\balpha\in\N^n_{t-1}}\Gamma_{\balpha}\x^{2\balpha}\,=\,\v_{t-1}(\x)^T\G_1\v_{t-1}(\x)\,,\quad\forall \x\,,\\
 \end{eqnarray*}
 where $\G_0,\G_1$ are diagonal positive definite matrices (i.e. $\G_0,\G_1\succ0$),  and so
 $(R+1)^t=\Delta+\theta\,\Gamma$. Next, let $\W$ be a  real symmetric matrix such that
 \[\v_t(\x)^T\W\,\v_t(\x)\,=\,\sum_{g\in G_t}g(\x)\,\v_{t-t_g}(\x)\I_{t-t_g}\v_{t-t_g}(\x)\,,\]
 where $\I_{t-t_g}$ is the $s(t-t_g)$-identity matrix.
 As $\G_0\succ0$ there exists $\delta>0$ such that $\G_0-\delta\,\W\succ0$. But then
 \begin{eqnarray*}
 (R+1)^t&=&\v_t(\x)^T(\G_0-\delta\,\W)\v_t(\x)+
 \delta\,\sum_{g\in G_t}g(\x)\,\v_{t-t_g}(\x)\I_{t-t_g}\v_{t-t_g}(\x)\\
 &&+\theta(\x)\,\v_{t-1}(\x)^T\G_1\v_{t-1}(\x)\,.\end{eqnarray*}
 Finally, as $\theta\in Q_1(G)$ and $\mathrm{deg}(g)\leq 2$ for all $g\in  G_1$, write
 \[\theta(\x)\,=\,\v_1(\x)^T\A_0\v_1(\x)+\sum_{g_0\neq g \in G_1}g(\x)\,a_g\,\]
 with $a_g\geq0$ and $\A_0\succeq0$,  to obtain
 \[(R+1)^t\,=\,\v_t(\x)^T(\G_0-\delta\,\W)\v_t(\x)+
 \delta\,\sum_{g\in G_t\setminus G_1}g(\x)\,\v_{t-t_g}(\x)\I_{t-t_g}\v_{t-t_g}(\x)\]
 \[+\sum_{g\in G_1;g\neq g_0}g(\x)\,[\delta\,\v_{t-1}(\x)\I_{t-1}\v_{t-1}(\x)+ a_g\,\v_{t-1}(\x)^T\G_1\v_{t-1}(\x)]\,,\]
 \[+\delta\,\v_{t}(\x)\I_{t}\v_{t}(\x)+ 
 \underbrace{\v_1(\x)^T\A_0\v_1(\x)\,\v_{t-1}(\x)^T\G_1\v_{t-1}(\x)}_{\v_t(\x)^T\V\,\v_t(\x)\,;\:\V\succeq0}\,,\]
 and therefore $(R+1)^t\in \mathrm{int}(Q_t(G))$.
\end{proof}

\end{document}